\documentclass{amsart}

\usepackage[font=small,labelfont=bf,up,textfont=it,up]{caption}

\usepackage[english]{babel}
\usepackage{tikz,graphicx}
\usetikzlibrary{matrix,arrows}

\usepackage{amsmath, amssymb, amsthm}

\newtheorem{theorem}{Theorem}[section]
\newtheorem{proposition}[theorem]{Proposition}

\newtheorem{corollary}[theorem]{Corollary}
\theoremstyle{definition}
\newtheorem{definition}[theorem]{Definition}
\newtheorem{remark}[theorem]{Remark}
\newtheorem{example}[theorem]{Example}

\newcommand{\cc}{\mathbb{C}}
\newcommand{\cp}{\mathbb{C}_p}
\newcommand{\nn}{\mathbb{N}}
\newcommand{\zz}{\mathbb{Z}}
\renewcommand{\mod}{\text{ mod}\ }

\makeatletter
\@namedef{subjclassname@2020}{\textup{2020} Mathematics Subject Classification}
\makeatother

\begin{document}

\author{Heiko Knospe, Lawrence C. Washington}
\title{Dirichlet Series Expansions of p-adic L-Functions}

\email{heiko.knospe@th-koeln.de, lcw@umd.edu}

\subjclass[2020]{Primary: 11R23. Secondary: 11R42, 11S80, 11M41}
\begin{abstract}
We study $p$-adic $L$-functions $L_p(s,\chi)$ for Dirichlet characters $\chi$. 
We show that $L_p(s,\chi)$ has a Dirichlet series expansion for each regularization parameter $c$ that is prime to $p$ and the conductor of $\chi$. 
The expansion is proved by transforming a known formula for $p$-adic $L$-functions and by controlling the limiting behavior. 
A finite number of Euler factors can be factored off  in a natural manner from the $p$-adic Dirichlet series.
We also provide an alternative proof of the expansion using $p$-adic measures and give an explicit formula for the values of the regularized Bernoulli distribution.
The result is particularly simple for $c=2$, where we obtain a Dirichlet series expansion that is similar to the complex case.  
\end{abstract}

\maketitle

\section{Introduction}
\label{intro}
Let $p$ be a prime, let $q=p$ if $p$ is odd and $q=4$ if $p=2$, and let $\chi$ be a Dirichlet character of conductor $f$.  
A {\em $p$-adic $L$-function} $L_p(s,\chi)$ for a Dirichlet character $\chi$ is a $p$-adic meromorphic function and an analogue of the complex $L$-function. 
For powers of the Teichm\"uller character $\omega$ of conductor $q$, one obtains the {\em $p$-adic zeta functions} $\zeta_{p,i}=L_p(s,\omega^{1-i})$, where $i=0,\, 1,\, \dots,\, p-2$ ($i=0, 1$ if $p=2$).
It is well known that $L_p(s,\chi)$ is identically zero for odd $\chi$. $p$-adic $L$-functions have a long history and the primary constructions going back to Kubota-Leopoldt \cite{kubota} and Iwasawa \cite{iwasawa} are via the interpolation of special values  of  complex $L$-functions. 
It can also be shown that $p$-adic $L$-functions are in fact Iwasawa functions. \\

It is well known that for $Re(s) > 0$, 
$$
(1-2^{1-s}) \zeta(s) = \sum_{n=1}^{\infty} \frac{(-1)^{n+1}}{n^s}$$
and, more generally, if $c\ge 2$ is an integer,  
$$(1-\chi(c)c^{1-s})L(s, \chi) = \sum_{n=1}^{\infty} \chi(n) \frac{a_{c,n}}{n^s},
$$
where $a_{c,n} = 1-c$ if $n\equiv 0 \mod c$ and $a_{c,n} = 1$ if $n\not\equiv 0 \mod c$.
In the following, we derive similar, but slightly different, expansions for $p$-adic $L$-functions.
 
An explicit formula for $L_p(s,\chi)$ is given in \cite{washington} (Theorem 5.11): let $F$ be any multiple of $q$ and $f$. Then $L_p(s,\chi)$ is a meromorphic function (analytic if $\chi \neq 1)$
on $\{ s \in \cp \ |\ |s| < qp^{-1/(p-1)} \}$ such that

\begin{equation}
 L_p(s,\chi) = \frac{1}{F} \frac{1}{s-1} \sum^F_{\substack{a=1 \\ p\, \nmid\, a}} \chi(a) \langle a \rangle^{1-s} \sum_{j=0}^{\infty} \binom{1-s}{j} \left( \frac{F}{a} \right)^j B_j .
  \label{wa511}
\end{equation}

In Section \ref{expansion}, we will use formula (\ref{wa511}) to derive a Dirichlet series expansion of $L_p(s,\chi)$. \\

 $p$-adic $L$-functions can be also be defined using distributions and measures.  Let $\chi$ have conductor $f=dp^m$ with $(d,p)=1$. Choose an integer $c \geq 2$, where 
 $(c,dp)=1$. Then there is a measure $E_{1,c}$ on $(\zz/d \zz)^{\times } \times\, \zz_p^{\times}$ (the {\em regularized Bernoulli distribution}) such that
 
  \begin{equation}
  -(1- \chi(c) \langle c \rangle^{1-s})  L_p(s, \chi) =  \int_{(\zz/d \zz)^{\times } \times\, \zz_p^{\times}} \chi \omega^{-1}(a)\langle a\rangle^{-s} \ dE_{1,c} 
  \label{wa122}
\end{equation}
(see \cite{washington} Theorem 12.2). In Section \ref{regularized}, we give an explicit formula for the values of $E_{1,c}$ and derive the Dirichlet series expansion from (\ref{wa122}).\\

The expansion is particularly simple for $c=2$, and this parameter can be used for $p \neq 2$ and Dirichlet characters with odd conductor. For this case we obtain similar results as in \cite{delbourgo2006dirichlet}, \cite{delbourgo2009}, and \cite{kim}. In Section \ref{examples}, we provide examples for different parameters $c$.

\section{Expansions of $p$-adic $L$-Functions}
\label{expansion}

First, we derive an approximation of $L_p(s,\chi)$ that is close to the original definition of Kubota-Leopoldt (see \cite{kubota}).\\

For $r \in \cc_p^{\times}$ we write $\delta(r)$ for a term with $p$-adic absolute value $\le |r|$.

\begin{proposition}\label{firstprop} Let $p$ be a prime number, $\chi$ an even Dirichlet character of conductor $f$, and $F$ a multiple of $q$ and $f$.  
For $s \in \cc_p$ with $|s| < qp^{-1/(p-1)}$, we have
\begin{equation}
 L_p(s,\chi) = \frac{1}{F} \frac{1}{s-1}  \sum^F_{\substack{a=1 \\ p\, \nmid\, a}} \chi(a) \langle a \rangle^{1-s}  + \delta(F/qp) .
 \label{kl}
 \end{equation}
 \label{klformula}
 \end{proposition}
 \begin{proof} We use formula (\ref{wa511}) above and look at the series $\sum_{j=0}^{\infty} \binom{1-s}{j} \left( \frac{F}{a} \right)^j B_j$. The first two terms are $1+(1-s) \frac{-F}{2a}$. We claim that the $p$-adic absolute value of the other terms ($j \geq 2$)  is  less than or equal to $ | (s-1)F^2/qp |$. 
To this end, we note that $|1/j!|\le p^{(j-1)/(p-1)}$ and 
$$
\left|\binom{1-s}{j}\right| \le  |1-s|\,  p^{(j-1)/(p-1)}(qp^{-1/(p-1)})^{j-1}= |1-s|\, q^{j-1} 
$$
since we assumed that $|s| < qp^{-1/(p-1)}$. Since $|F|\leq \frac{1}{q}$, $|a|=1$, and $|B_j| \leq p$, we obtain 
 $$
\left| \binom{1-s}{j} \left( \frac{F}{a} \right)^j B_j \right| \leq |1-s| \,  q^{j-1} q^{2-j} |F|^2\, p = |1-s| \, |F|^2\, q p. $$
 Then (\ref{wa511}) implies
 $$ L_p(s, \chi) = \frac{1}{F} \frac{1}{s-1}  \sum^F_{\substack{a=1 \\ p\, \nmid\, a}} \chi(a) \langle a \rangle^{1-s}  + \frac{1}{2} \sum^F_{\substack{a=1 \\ p\, \nmid\, a}} \chi \omega^{-1}(a) \langle a \rangle^{-s} +  \delta(F/qp) .$$
It remains to show that the second sum can be absorbed into $\delta(F/qp)$. We have
\begin{align*} 
\sum^F_{\substack{a=1 \\ p\, \nmid\, a}} \chi \omega^{-1}(a) \langle a \rangle^{-s} & 
= \sum^F_{\substack{b=1 \\ p\, \nmid\, b}} \chi \omega^{-1}(F-b) \langle F-b \rangle^{-s} \\
& = - \sum^F_{\substack{b=1 \\ p\, \nmid\, b}} \chi \omega^{-1}(b) \langle b-F \rangle^{-s} \\
& = - \sum^F_{\substack{b=1 \\ p\, \nmid\, b}} \chi \omega^{-1}(b) \langle b \rangle^{-s} + \delta(F/qp^{-1/(p-1)}).
\end{align*}
The last step can be justified by noting that
$$
\frac{\langle b-F\rangle^{-s}}{\langle b\rangle^{-s}} = \left(1-\frac{F}{b}\right)^{-s}  = 1 +\sum_{j=1}^{\infty} \binom{-s}{j} \left(\frac{-F}{b}\right)^j
= 1 + \delta(F/qp^{-1/(p-1)}),
$$ 
since $|s|< qp^{-1/(p-1)}$ (this is the same estimate as earlier, without the presence of the Bernoulli number).
This proves  the proposition.
 \end{proof}

\begin{remark} For $F=fp^n$ and $n \rightarrow \infty$, formula (\ref{kl}) gives the original definition of $L_p(s,\chi)$ by Kubota and Leopoldt (see \cite{kubota}).
\end{remark}

\begin{remark} Suppose that $p\neq 2$. Then the error term in the above Proposition (as well as in the following Theorem \ref{expansionthm}) can be improved to $\delta(F/p^{2-(p-2)/(p-1)})$.
First we note that $B_j=0$ for odd $j \geq 3$. 
By the von Staudt--Clausen Theorem  (see \cite{washington} 5.10), we have for even $j \geq 2$: $|B_j| =p$ iff $(p-1) \mid j$, and otherwise $|B_j| \leq 1$.  
Furthermore, $|1/j!| =  p^{(j-S_j)/(p-1)}$, where $S_j$ is the sum of the digits of $j$, written to the base $p$ (see \cite{koblitz}). 
Since $j \equiv S_j \mod (p-1)$, $j \equiv 0 \mod (p-1)$ is equivalent to $S_j \equiv 0 \mod (p-1)$. We conclude that $|B_j |=p$ yields $S_j \geq p-1$ 
and $|1/j!| \leq  p^{(j-1)/(p-1) } p^{-(p-2)/(p-1)}$. This implies the above error term. We also see that this error term cannot be further improved.$\hfill\lozenge$
\end{remark}

Now we give the Dirichlet expansion of $L_p(s,\chi)$. 
 For $m \in \nn$, we denote by $\{x\}_{m}$ the unique representative of $x \mod m\zz$ between $0$ and $m-1$.
 
 \begin{theorem} Let $p$ be a prime number, $\chi$ be an even Dirichlet character of conductor $f$, and  $F$ 
a multiple of $q$ and $f$. Let $c>1$ be an integer satisfying $(c,F)=1$.
For $a \in \zz$, define
$$\epsilon_{a,c,F} = \frac{c-1}{2} - \{ -a F^{-1} \}_{c}
\, \in\left \{ - \frac{c-1}{2} ,\, - \frac{c-1}{2} + 1,\, \dots,\ \frac{c-1}{2}\right \}.$$
Then we have for $s \in \cc_p$ with $|s| < qp^{-1/(p-1)}$ the formula

 $$ -(1- \chi(c) \langle c \rangle^{1-s})  L_p(s, \chi) = \sum^F_{\substack{a=1 \\ p\, \nmid\, a}} \chi \omega^{-1}(a) \langle a \rangle^{-s} \epsilon_{a,c,F} + \delta(F/qp) .$$
 
 \label{expansionthm}
 \end{theorem} 
\begin{proof} Use (\ref{kl}) with $cF$ in place of $F$, and subtract $\chi(c) \langle c \rangle^{1-s}$ times (\ref{kl}) with $F$, to obtain

\begin{align}
\begin{split}
(1- \chi(c) \langle c \rangle^{1-s})  L_p(s, \chi) & = \frac{1}{cF} \frac{1}{s-1}  \sum^{cF}_{\substack{a=1 \\ p\, \nmid\, a}} \chi(a) \langle a \rangle^{1-s} \\ & - 
\frac{1}{F} \frac{1}{s-1}  \sum^F_{\substack{a=1 \\ p\, \nmid\, a}} \chi(ac) \langle ac \rangle^{1-s}  + \delta(F/qp).
\label{sum}
\end{split}
\end{align}

Let $0<a_0 < F$ with $(a_0,\, p)=1$. Since we assumed $(c,\, F)=1$ and $p \mid F$, there is a unique number of the form $a_0 c$ with $0<a_0c < cF$ and $(a_0 c,\, p)=1$ in each congruence class modulo $F$ relatively prime to $p$. The first sum in (\ref{sum}) can be written as
\begin{align*}
& \frac{1}{cF} \frac{1}{s-1}  \sum^{F}_{\substack{a_0=1 \\ p\, \nmid\, a_0}} \chi(a_0 c) \langle a_0 c \rangle^{1-s}  \left( \sum_{\substack{a=1 \\ a \equiv a_0 c \mod F}}^{cF} \left \langle 1+ \frac{a-a_0 c}{a_0 c} \right\rangle^{1-s} \right) \\
= & \frac{1}{cF} \frac{1}{s-1}  \sum^{F}_{\substack{a_0=1 \\ p\, \nmid\, a_0}} \chi(a_0 c) \langle a_0 c \rangle^{1-s}  \left( \sum_{\substack{a=1 \\ a \equiv a_0 c \mod F}}^{cF}  \left(1+ (1-s) \frac{a-a_0 c}{a_0 c} \right)  \right) + \delta(F/q) .
\end{align*}
Note that $| \frac{a-a_0 c}{a_0 c} | \leq |F|$, so this is the same type of estimate used in the proof of Proposition \ref{firstprop}. Subtracting the second sum in (\ref{sum}) yields
\begin{align*} 
(1- \chi(c)& \langle c \rangle^{1-s})  L_p(s, \chi)\\ & = \frac{-1}{cF} \sum^{F}_{\substack{a_0=1 \\ p\, \nmid\, a_0}} \chi(a_0 c) \langle a_0 c \rangle^{1-s}  \left( \sum_{\substack{a=1 \\ a \equiv a_0 c \mod F}}^{cF}  \frac{a-a_0 c}{a_0 c}  \right) + \delta(F/qp) \\
& = \frac{-1}{c} \sum^{F}_{\substack{a_0=1 \\ p\, \nmid\, a_0}} \chi \omega^{-1}(a_0 c) \langle a_0 c \rangle^{-s}  \left( \sum_{\substack{a=1 \\ a \equiv a_0 c \mod F}}^{cF}  \frac{a-a_0 c}{F}  \right) + \delta(F/qp) .
\end{align*}
We compute the inner sum. Let $b=\{a_0 c\}_F$. Then $a_0 c = b + \{ -F^{-1} b \}_c\, F$, since the latter sum is congruent to $b$ modulo  $F$ and congruent to $0$ modulo $c$. If $a$ satisfies $a \equiv a_0 c \mod F$ and $0<a<cF$, then $a=b+jF$ with $0\leq j <c$. Hence
$$ \sum_{\substack{a=1 \\ a \equiv a_0 c \mod F}}^{cF}  \frac{a-a_0 c}{F} = \sum_{j=0}^{c-1} (j -  \{ - F^{-1} b \}_c) = c\ \epsilon_{b,c,F}\ .$$
Since $b \equiv a_0 c \mod F$, we have $\chi \omega^{-1}(b) \langle b \rangle^{-s} = \chi \omega^{-1}(a_0 c) \langle a_0 c \rangle^{-s} + \delta(F/q)$ by the same estimate as earlier, so
$$ -(1- \chi(c) \langle c \rangle^{1-s})  L_p(s, \chi) = \sum^F_{\substack{b=1 \\ p\, \nmid\, b}} \chi \omega^{-1}(b) \langle b \rangle^{-s} \epsilon_{b,c,F} + \delta(F/qp).$$
This completes the proof.
\end{proof} 

We can take the limit of $F=fp^n$ as $n \rightarrow \infty$ and obtain: 

\begin{corollary} 
Let $p$ be a prime number, $\chi$ an even Dirichlet character of conductor $f$, and $c>1$ an integer satisfying $(c,pf)=1$. Then
we have for $s \in \cc_p$ with $|s| < qp^{-1/(p-1)}$,

$$  -(1- \chi(c) \langle c \rangle^{1-s})  L_p(s, \chi) = \lim_{n \rightarrow \infty}  \sum^{fp^n}_{\substack{a=1 \\ p\, \nmid\, a}} \chi \omega^{-1}(a)
\frac{ \epsilon_{a,c,fp^n} }{ \langle a \rangle^{s}}.$$

\label{limit}
\end{corollary}

The next Theorem shows that a finite number of Euler factors can be factored off in a similar way as in \cite{eulerfactors}, where a {\em weak Euler product} was obtained.  The main statement is that the remaining Dirichlet series has the expected form, similar to the complex case. 

 \begin{theorem} 
 Let $p$ be a prime number and let $\chi$ be an even Dirichlet character of conductor $f$. Let
$S$ be any finite (or empty) set of primes not containing $p$ and set $S^+ = S \cup \{p\}$. Let
$F$ be  
a multiple of $q$, $f$ and all primes in $S$. Let $c>1$ be an integer satisfying $(c,F)=1$.
Then we have for $s \in \cc_p$ with $|s| < qp^{-1/(p-1)}$ the formula

\begin{align*}  -(1- \chi(c) \langle c \rangle^{1-s}) \cdot \prod_{l \in S} (1-\chi \omega^{-1}(l) \langle l \rangle^{-s}    )\ \cdot\  L_p(s, \chi)  = \\
 \sum^F_{\substack{a=1 \\ (a,S^+)=1}} \chi \omega^{-1}(a) \frac{\epsilon_{a,c,F}}{\langle a \rangle^{s} } \  + \ \delta(F/qp) . \end{align*}
 
\end{theorem} 

\begin{proof} We prove the statement by induction on $|S|$. By Theorem \ref{expansionthm}, the formula is true for $S=\varnothing$. Now assume the formula is true for $S$, and $l \neq p$ is a prime with $l \notin S$ and $(c,l)=1$. 
It suffices to prove the following formula:

\begin{align} 
\begin{split}
(1-\chi \omega^{-1}(l) \langle l \rangle^{-s}   )  \sum^{F}_{\substack{a=1 \\ (a,S^+)=1}} \chi \omega^{-1}(a) \langle a \rangle^{-s} \epsilon_{a,c,F}\  = \\
\sum^{lF}_{\substack{a=1 \\ (a,S^+ \cup \{l\})= 1}} \chi \omega^{-1}(a) \langle a \rangle^{-s} \epsilon_{a,c,lF}\
+ \ \delta(F/qp) .
\label{claim}
\end{split}
\end{align}
Note that $|1-\chi \omega^{-1}(l) \langle l \rangle^{-s} | \leq 1$ and $|lF|=|F|$, so we can keep the error term.
We can use $lF$ in place of $F$ and write the left side of (\ref{claim}) as
\begin{align}  \sum^{lF}_{\substack{a=1 \\ (a,S^+)=1}} \chi \omega^{-1}(a) \langle a \rangle^{-s} \epsilon_{a,c,lF} \ -
  \sum^{F}_{\substack{a=1 \\ (a,S^+)=1}} \chi \omega^{-1}(la) \langle la \rangle^{-s} \epsilon_{a,c,F} + \delta(F/qp) .
  \label{difference}
  \end{align}
  Now we have
  $$ \epsilon_{la,c,lF} = \frac{c-1}{2} - \{ -la (lF)^{-1} \}_{c} = \frac{c-1}{2} - \{ -a F^{-1} \}_{c} = \epsilon_{a,c,F}. $$
  Thus (\ref{difference}) is equal to 
\begin{align*}  
  \sum^{lF}_{\substack{a=1 \\ (a,S^+)=1}} \chi \omega^{-1}(a) \langle a \rangle^{-s} \epsilon_{a,c,lF} \ & - 
  \sum^{F}_{\substack{a=1 \\ (a,S^+)=1 }} \chi \omega^{-1}(la) \langle la \rangle^{-s} \epsilon_{la,c,lF} + \delta(F/qp) \\
  & = \sum^{lF}_{\substack{a=1 \\ (a,S^+)= 1 \\ l \,\nmid\, a}} \chi \omega^{-1}(a) \langle a \rangle^{-s} \epsilon_{a,c,lF}\
+ \ \delta(F/qp),
\end{align*}
 which shows equation (\ref{claim}). 
\end{proof}

\begin{remark} What happens if $S$ contains more and more primes? It is well known that the Euler product does not converge $p$-adically  (see \cite{delbourgo2009}), since the factors \linebreak $(1-\chi \omega^{-1}(l) \langle l \rangle^{-s}   )$ have absolute value $\leq 1$ and do not converge to $1$ as $l \rightarrow \infty$. Furthermore, there are infinitely many primes $l$ with $\chi \omega^{-1} (l) = 1$ and  
$ (1 - \langle l \rangle^{-s})^{-1}$ has a pole at $s=0$. We have for $l \neq p$ and  $|s| < qp^{-1/(p-1)}$,
$$ 1 - \langle l \rangle^{-s} = - \sum_{j=1}^{\infty} \binom{-s}{j} ( \langle l \rangle -1 )^j .$$
The $p$-adic absolute value of each term of the above series is less than
$$  (qp^{-1/(p-1)})^j p^{(j-1)/(p-1)} q^{-j} = p^{-1/(p-1)} < 1 .$$
Hence the product $\prod_{l \in S} (1-\chi \omega^{-1}(l) \langle l \rangle^{-s}  )$ 
approaches $0$ as $S$ expands to include all primes. 
\end{remark}

\section{Regularized Bernoulli Distributions}
\label{regularized}

Let $p$ be a prime number and let $d$ be a positive integer with $(d,p)=1$. Define $X_n = (\zz/dp^n \zz)$ and $X = \varprojlim X_n \cong \zz/d \zz \times \zz_p$.
Let $k \geq 1$ be an integer. 
Then the  {\em Bernoulli distribution} $E_k$ on $X$ is defined by 
$$ E_k( a + dp^ n X) = (dp^n)^{k-1} \frac{1}{k} B_k \left(\frac{ \{ a \}_{d p^n} } {d p^n} \right) , $$
where $B_k(x)$ is the $k$-th Bernoulli polynomial and $B_k=B_k(0)$ are the Bernoulli numbers (see \cite{koblitz}, \cite{lang}). For $k=1$, one has  $B_1(x)= x - \frac{1}{2}$.
Choose $c \in \zz$ with $c \neq 1$ and $(c,dp)=1$.     
Then the {\em regularization} $E_{k,c}$ of $E_k$ is defined by
$$ E_{k,c} ( a + dp^n X ) = E_k ( a + dp^n X) - c^k E_k \left( \left\{ \frac{a}{c}  \right\}_{d p^n} + d p^n X \right) .$$
One shows that the regularized Bernoulli distributions $E_{k,c}$ are {measures} (see \cite{lang}). In the following, we consider only $k=1$; the cases $k \geq 2$ are similar.

\begin{theorem} Let $p$ be a prime, $c,\, d \in \nn$, and $c \geq 2$ such that $(c,dp)=1$. Let $X$ be as above, and let $E_{1,c}$ be the regularized Bernoulli distribution on $X$.
For  $a \in \{0,1,\dots, dp^n-1\}$, we have
$$ E_{1,c}  ( a + dp^n X) = \frac{c-1}{2} - \{ -a (dp^n)^{-1} \}_{c} = \epsilon_{a,c,dp^n}.$$

\label{measure}
\end{theorem}

\begin{proof} 
By definition,
$$E_{1,c}(a+dp^n X) = E_1(a + dp^n X) - c E_1(c^{-1}a + dp^n X) = \frac{a}{dp^n} - \frac{1}{2} - c \left(\frac{ \{c^{-1} a \}_{dp^n}} {dp^n} \right) + \frac{c}{2}.  $$
We give the standard representative of $c^{-1} a \mod dp^n$: 
$$ \{c^{-1} a \}_{dp^n} = \frac{  \{ -a (dp^n)^{-1} \}_{c}\ dp^n + a}{c} $$
Note that the numerator is divisible by $c$, since $ \{ -a (dp^n)^{-1} \}_{c}\ dp^n \equiv -a \mod c$. Hence the quotient is an integer between $0$ and $dp^n-1$.  Furthermore, the numerator is congruent to $a$ modulo $dp^n$, and so the quotient has the desired property.
We obtain
 $$ E_{1,c}(a+dp^n X) = \frac{a}{dp^n} +  \frac{c-1}{2} - \frac{  \{ -a (dp^n)^{-1} \}_{c}\ dp^n + a}{dp^n} = \frac{c-1}{2} - \{ -a (dp^n)^{-1} \}_{c} $$
 which is the assertion.
\end{proof}

Now the Dirichlet series expansion in Corollary \ref{limit} follows from Theorem \ref{measure} and the integral formula (\ref{wa122}).

\section{Expansions for different regularization parameters}
\label{examples}

We look at the coefficients  $\epsilon_{a,c,dp^n}$ for different parameters $c$ and the resulting Dirichlet series expansions. The following observation follows directly from the definition.
\begin{remark} The sequence of values $E_{1,c}  ( a + dp^n X)  = \epsilon_{a,c,dp^n}$ for  $a=$ $0$,\linebreak $1$, $2$, $\dots$, $dp^n-1$ is periodic with period $c$. The sequence begins with $\frac{c-1}{2}$ and continues with a permutation of $\frac{c-3}{2}, \dots, - \frac{c-1}{2}$. If we restrict to values of $n$ such that $dp^n$ lies in a fixed congruence class modulo $c$, then the values do not change as $n \rightarrow \infty$. $\hfill$ $\lozenge$
\end{remark}

The measure $E_{1,c}$ and the Dirichlet series expansion are particularly simple for $c=2$.
Note that we assumed that $d$ and $p$ are odd in this case.
If $a$ is even, then $ \{ -a (dp^n)^{-1} \}_{2} = 0$ and 
$$ E_{1,2}(a+dp^n X)  =  \epsilon_{a,2,dp^n} = \frac{1}{2}.$$
If $a$ is odd, then $-a(dp^n)^{-1}$ is odd, $\{ -a (dp^n)^{-1} \}_{2} = 1$ and 
$$ E_{1,2}(a+dp^n X) =  \epsilon_{a,2,dp^n}   = - \frac{1}{2}. $$
Hence $E_{1,2}$ is up to the factor $\frac{1}{2}$ equal to the following simple measure:
\begin{definition} Let $ p \neq 2$ be a prime, and let $X \cong \zz/d\zz \times \zz_p$ be as above. Then 
$$ \mu(a + dp^n X) = (-1)^{\{a\}_{dp^n}} $$
defines a measure on $X$. We call $\mu$ the {\em alternating measure}, since the measure of all clopen balls is $\pm 1$.\hfill$\lozenge$ 
\end{definition}
The corresponding integral is also called the {\em fermionic} $p$-adic integral (see \cite{kim}).\\

Now we obtain the following Dirichlet series expansion from Corollary \ref{limit}.

\begin{corollary} 
Let $p \neq 2$ be a prime number, and let $\chi$ be an even Dirichlet character of odd conductor $f$. Then
we have for $s \in \cc_p$ with $|s| < p^{(p-2)/(p-1)}$,
$$  (1- \chi(2) \langle 2 \rangle^{1-s})  L_p(s, \chi) = \lim_{n \rightarrow \infty} \frac12 \sum^{fp^n}_{\substack{a=1 \\ p\, \nmid\, a}}(-1)^{a+1} \chi \omega^{-1}(a) \frac{1}{\langle a \rangle^{s}} .$$
For $\chi=\omega^{1-i}$ and odd $i=1,\, \dots,\, p-2$, we obtain the branches of the $p$-adic zeta function:
$$ \zeta_{p,i}(s) = L_p(s,\, \omega^{1-i}) = \frac{1}{1-\omega(2)^{1-i} \langle 2 \rangle^{1-s}}  \cdot  \lim_{n \rightarrow \infty} \frac12\sum_{\substack{a=1 \\ p\, \nmid\, a }}^{p^n} (-1)^{a+1}\omega(a)^{-i}\frac{1}{ \langle a \rangle^{s}} $$

\label{limit2}
\end{corollary}

\begin{remark}
Dirichlet series expansions of $p$-adic $L$-functions were studied by D. Delbourgo  in \cite{delbourgo2006dirichlet} and  \cite{delbourgo2009}. 
He considers Dirichlet characters $\chi$ satisfying $(p,\, 2 f  \phi(f))=1$ and their Teichm\"uller twists. 
We obtain the same expansion for $c=2$ and $\chi= \omega^{1-i}$. However, we require $(c,\, fp)=1$ and use other methods for the proof.  

Similar expansions for a slightly different $p$-adic $L$-function using a fermionic $p$-adic integral (i.e., $c=2$)  were also obtained by M.-S. Kim and S. Hu (see \cite{kim}).

\end{remark}

\begin{example}
We look at the case $c=3$. The sequence of values $\epsilon_{a,3,dp^n}$ is periodic with period $3$. If $dp^n \equiv 1 \mod 3$, then the sequence is $1,\ -1,\ 0,\, \dots$ . If $dp^n \equiv 2 \mod 3$, then we obtain the sequence $1,\ 0,\ -1,\, \dots$ .
\end{example}

\begin{corollary}  
Let $p$ be a prime number, and let $\chi$ be an even Dirichlet character of conductor $f=dp^m$ such that $(3,\, dp)=1$.
If $d \equiv 1 \mod 3$, then define a sequence $\epsilon_0=1$, $\epsilon_1=-1$, $\epsilon_2=0$, $\dots$ with period $3$. Otherwise, set $\epsilon_0=1$, $\epsilon_1=0$, $\epsilon_2=-1$ and extend it with period $3$.
Then
we have for $s \in \cc_p$ with $|s| < qp^{-1/(p-1)}$,
$$  -(1- \chi(3) \langle 3 \rangle^{1-s})  L_p(s, \chi) = \lim_{n \rightarrow \infty}  \sum^{dp^{2n}}_{\substack{a=1 \\ p\, \nmid\, a}} \chi \omega^{-1}(a)\frac{\epsilon_a}{ \langle a \rangle^{s} }.$$
\label{limit3}
\end{corollary}

\begin{example} For $c=5$, we get a periodic sequence with period $5$ and we have $\epsilon_{a,5,dp^n}=2$ for $a \equiv 0 \mod 5$. The next four coefficients are a permutation of the values $-2$, $-1$, $0$ and $1$, depending on the class of $dp^n \mod 5$.
\end{example}

\begin{example} Let $c=7$. Then $\epsilon_{0,7,dp^n} = 3$. Now suppose, for example, that 
$dp^n \equiv 3 \mod 7$. Then $(dp^n)^{-1} \equiv 5 \mod 7$. This yields the values
$$  \epsilon_{1,7,dp^n} = 1,\ \epsilon_{2,7,dp^n} = -1,\ \epsilon_{3,7,dp^n} = -3,\ \epsilon_{4,7,dp^n} = 2,\
  \epsilon_{5,7,dp^n} = 0,\ \epsilon_{6,7,dp^n} = -2, $$
and these are extended with period $7$. 
\end{example}

\noindent{\bf Acknowledgements. }
The first author thanks Daniel Delbourgo for hints to his work and helpful conversations.

\bibliographystyle{plain}

\end{document}